\documentclass[11pt]{article}

\usepackage{bbm}
\usepackage[T1]{fontenc}
\usepackage[utf8]{inputenc}
\usepackage[english]{babel}

\usepackage{amsmath,amssymb,amsthm,mathtools}
\usepackage{enumitem}
\usepackage[margin=1in]{geometry}
\usepackage{hyperref}

\theoremstyle{definition}
\newtheorem{definition}{Definition}[section]
\newtheorem{remark}[definition]{Remark}
\newtheorem{notation}[definition]{Notation}

\theoremstyle{plain}
\newtheorem{theorem}[definition]{Theorem}
\newtheorem{proposition}[definition]{Proposition}
\newtheorem{lemma}[definition]{Lemma}
\newtheorem{corollary}[definition]{Corollary}

\numberwithin{equation}{section}

\title{Rigid many-one degrees contain infinite antichains of $1$-degrees}

\author{
Patrizio Cintioli\\
Mathematics Division, School of Science and Technology\\
University of Camerino, Italy\\
\texttt{patrizio.cintioli@unicam.it}
}

\date{} 

\begin{document}

\maketitle

\begin{abstract}
Odifreddi asked whether every non-irreducible many-one degree must contain an infinite antichain of
one-one degrees. Positive answers are known for computably enumerable many-one degrees (Degtev) and,
more recently, for many-one degrees admitting a $\Delta^0_2$ representative (Batyrshin). In this note we
isolate a rigidity principle behind these phenomena.

Call a set $A\subseteq\omega$ \emph{$m$-rigid} if every total computable $m$-autoreduction of $A$ is eventually
the identity. We prove that if $A$ is $m$-rigid, then its many-one degree $\deg_m(A)$ contains an infinite
antichain of $1$-degrees. The proof uses a uniform duplication construction: for each computable parameter
$S$ we define $B_S\equiv_m A$ so that any injective reduction $B_S\le_1 B_T$ induces an $m$-autoreduction of
$A$ and therefore forces $S\subseteq^{*}T$. Choosing an almost-inclusion infinite antichain of computable sets yields
the desired infinite $1$-antichain inside $\deg_m(A)$.

As applications, Jockusch's rigidity theorem implies that every $1$-generic set is $m$-rigid, giving a
comeager family of positive instances. Moreover, $m$-rigidity holds with Lebesgue measure $1$ (indeed,
every Martin-L\"of random real is $m$-rigid). Consequently, Odifreddi's Question~5 has a positive answer
\emph{with probability $1$} for a fair-coin random $A\in 2^\omega$; any counterexample (if it exists) is
confined to a null set (and, by genericity, also to a meager set).
\end{abstract}
\medskip


\medskip

\section{Introduction}

Many-one reducibility $\le_m$ and one-one reducibility $\le_1$ induce the corresponding degree structures of
$m$-degrees and $1$-degrees. Every $m$-degree may split into many $1$-degrees, and a basic structural problem
is to understand how complicated this splitting can be.

\medskip
\noindent
\textbf{Odifreddi's Question~5.}
Odifreddi asked whether every many-one degree containing more than one one-one degree must contain an
\emph{infinite antichain} of one-one degrees \cite{Odifreddi1981,Odifreddi1999}.
Equivalently: if an $m$-degree is
not irreducible, must it contain infinitely many pairwise $\le_1$-incomparable representatives?
\smallskip

\noindent\emph{Note.} We tacitly restrict to \emph{noncomputable} $m$-degrees, since the nontrivial computable
$m$-degree is non-irreducible but contains no infinite $1$-antichains.

\paragraph{Previous positive results.}
The question remains open in full generality, but it has been answered positively in several important settings.
In the computably enumerable case, Degtev proved a trichotomy: every c.e.\ many-one degree is either recursive,
or consists of a single one-one degree, or contains an infinite antichain of one-one degrees \cite{Degtev1973}.
Degtev later refined the picture by constructing, for each $n$, a c.e.\ many-one degree with a least one-one degree
and exactly $n$ minimal one-one degrees above it \cite{Degtev1976}.
More recently, Batyrshin extended the infinite-antichain phenomenon beyond the c.e.\ case by showing that
every $\Delta^0_2$ many-one degree that contains more than one one-one degree already contains an infinite antichain
of one-one degrees \cite{Batyrshin2021}.

\paragraph{A rigidity principle.}
In this note we isolate a simple structural property that, by itself, forces the existence of infinite one-one
antichains inside a fixed many-one degree.
We call a set $A\subseteq\omega$ \emph{$m$-rigid} if every total computable $m$-autoreduction of $A$ is eventually the
identity. Our main theorem shows that $m$-rigidity already implies strong $1$-degree complexity:

\smallskip
\noindent
\emph{If $A$ is $m$-rigid, then the many-one degree $\deg_m(A)$ contains an infinite antichain of $1$-degrees.}

\smallskip
\noindent
The proof is based on a uniform ``duplication'' construction: from each computable parameter $S$ we build a set
$B_S\equiv_m A$ such that any injective reduction $B_S\le_1 B_T$ induces an $m$-autoreduction of $A$, hence forces
$S\subseteq^{*}T$. Choosing an almost-inclusion infinite antichain of computable sets yields the desired $1$-antichain in
$\deg_m(A)$.

\paragraph{Typicality: category and measure.}
The rigidity viewpoint immediately connects Odifreddi's question with two standard notions of typicality on Cantor
space $2^\omega$. Beyond Baire category, the rigidity viewpoint also yields a measure-theoretic strengthening.
On the one hand, by a classical theorem of Jockusch, every $1$-generic set is $m$-rigid
\cite{Jockusch1980}; since $1$-generics form a comeager set, this yields a comeager family of positive instances.
On the other hand, we show that every Martin-L\"of random set is $m$-rigid; since Martin-L\"of randoms form a measure $1$ class, this yields the following almost sure consequence: Odifreddi's Question~5 holds \emph{with probability $1$} for a fair-coin random $A\in 2^\omega$. In particular, any counterexample (if it exists) must lie in a null set (and also in a meager set).

\paragraph{Outline.}
Section~\ref{sec:rigid-trap} proves the rigidity trap theorem for arbitrary $m$-rigid sets.
Section~\ref{sec:typicality} establishes $m$-rigidity for $1$-generic and Martin-L\"of random sets and derives the
almost-everywhere consequences.

\section{Preliminaries}

\noindent\textbf{Conventions.}
Our notation and basic background from computability theory, degree structures,
and algorithmic randomness are standard. For general references see, e.g.,
Odifreddi~\cite{OdifreddiCRT} and Soare~\cite{Soare1987}, and for algorithmic randomness
(in particular Martin-L\"of tests and Martin-L\"of randomness) see
Downey--Hirschfeldt~\cite{DH2010} and Nies~\cite{Nies2009}.
We write $\omega=\{0,1,2,\dots\}$ and identify each set $A\subseteq\omega$ with its characteristic function
in Cantor space $2^\omega$. Thus we write $A(n)=1$ iff $n\in A$, and $A(n)=0$ otherwise.
Finite binary strings are elements of $2^{<\omega}$. For $\sigma,\tau\in 2^{<\omega}$ we write
$\sigma\preceq\tau$ if $\sigma$ is an initial segment of $\tau$, and for $X\in 2^\omega$ we write
$\sigma\prec X$ if $\sigma$ is an initial segment of $X$.

For $\sigma\in 2^{<\omega}$ let
\[
[\sigma] := \{X\in 2^\omega : \sigma\prec X\}
\]
be the basic clopen (cylinder) set determined by $\sigma$. For $W\subseteq 2^{<\omega}$ set
\[
[W] := \bigcup_{\sigma\in W}[\sigma],
\]
the open set generated by $W$ (effectively open if $W$ is c.e.). A set $D\subseteq 2^{<\omega}$ is \emph{dense} if for every
$\tau\in 2^{<\omega}$ there exists $\sigma\succeq\tau$ with $\sigma\in D$ (equivalently, $[D]$ is dense open).

We write $\mathrm{REC}$ for the class of computable subsets of $\omega$.

\begin{definition}[Many-one and one-one reducibility]
For sets $A,B\subseteq\omega$:
\begin{itemize}[leftmargin=2em]
\item $A\le_1 B$ if there exists a total computable \emph{injective} function $f:\omega\to\omega$ such that
$\forall x\,\bigl(x\in A \leftrightarrow f(x)\in B\bigr)$.
\item $A\le_m B$ if there exists a total computable function $f:\omega\to\omega$ such that
$\forall x\,\bigl(x\in A \leftrightarrow f(x)\in B\bigr)$.
\end{itemize}
\end{definition}

\begin{definition}[Degrees]
Define $A\equiv_1 B$ iff $A\le_1 B$ and $B\le_1 A$; similarly $A\equiv_m B$ iff $A\le_m B$ and $B\le_m A$.
The $1$-degree and $m$-degree of $A$ are
\[
\deg_1(A):=\{X\subseteq\omega : X\equiv_1 A\},\qquad
\deg_m(A):=\{X\subseteq\omega : X\equiv_m A\}.
\]
An $m$-degree is \emph{irreducible} if it contains exactly one $1$-degree.
\end{definition}

\begin{definition}[$m$-autoreductions and $m$-rigidity]
Let $A\subseteq\omega$. A total computable function $k:\omega\to\omega$ is an \emph{$m$-autoreduction of $A$} if
\[
(\forall x)\ \bigl[x\in A \iff k(x)\in A\bigr].
\]
We say that $A$ is \emph{$m$-rigid} if every $m$-autoreduction $k$ of $A$ is eventually the identity, i.e.\
$k(x)=x$ for all but finitely many $x$.
\end{definition}
\smallskip
\noindent\emph{Note.} Every $m$-rigid set is necessarily infinite and coinfinite.

\begin{lemma}
\label{lem:mrigid-coinfinite}
If $A$ is $m$-rigid, then $A$ is bi-immune.
\end{lemma}

\begin{proof}
If $A$ had an infinite computable subset $S\subseteq A$, define $k(x)$ to be the next element of $S$ above $x$
when $x\in S$, and $k(x)=x$ otherwise. Then $k$ is a total computable $m$-autoreduction of $A$ with
$k(x)\neq x$ for infinitely many $x\in S$, contradicting $m$-rigidity. Hence $A$ has no infinite computable
subset. 
Note that $k^{-1}(A)=A\Leftrightarrow k^{-1}(\overline{A})=\overline{A}$, hence $\overline{A}$ is $m$-rigid whenever $A$ is.
Applying the same argument to $\overline{A}$ shows that $\overline{A}$ has no infinite computable subset either, so $A$ is bi-immune.
\end{proof}

\begin{remark}
Lemma \ref{lem:mrigid-coinfinite} shows that every $m$-rigid set is bi-immune.
Consequently, no computably enumerable set can be $m$-rigid. 
This implies that the rigidity principle isolated here is fundamentally orthogonal to the mechanism driving infinite antichains in the c.e. degrees (Degtev's theorem).
\end{remark}

\begin{notation}
For $S,T\subseteq\omega$ we write $S \subseteq^* T$ if $S\setminus T$ is finite.
Equivalently, $S \subseteq^* T$ means that $x\in S \Rightarrow x\in T$ for all but finitely many $x$.

\end{notation}

\begin{remark}[Many-one degrees refine Turing degrees]
If $A\le_m B$ via a total computable function $f$, then $A\le_T B$ (compute $f(x)$ and query $B$). Hence
$A\equiv_m B$ implies $A\equiv_T B$, so each $m$-degree is contained in a single Turing degree.

In particular, in the duplication construction of Theorem~\ref{thm:mrigid-antichain} we always have
$B_S\equiv_m A$, and therefore $B_S\equiv_T A$ for every computable parameter $S$.
Thus our infinite $\le_1$-antichains are produced \emph{within the same Turing degree} as $A$, and they do not
lower the arithmetical/hyperarithmetical complexity of the base real: membership in any class closed under
$\equiv_T$ (e.g.\ $\Delta^0_n$ or $\Delta^1_1$) is uniform on $\deg_m(A)$.
Equivalently, if $A$ is not hyperarithmetical, then $\deg_m(A)$ contains no hyperarithmetical representative.
This is the reason why the typicality results in Section~\ref{sec:typicality} can be interpreted directly at the
level of $m$-degrees.
\end{remark}

\begin{notation}[Lebesgue measure on $2^\omega$]
We write $\mu$ for the usual fair-coin product measure on $2^\omega$, determined by
$\mu([\sigma])=2^{-|\sigma|}$ for every $\sigma\in 2^{<\omega}$.
\end{notation}

\begin{definition}[Martin-L\"of randomness]
A \emph{Martin-L\"of test} is a uniformly c.e.\ sequence $(U_n)_{n\in\omega}$ of open subsets of $2^\omega$
such that $\mu(U_n)\le 2^{-n}$ for all $n$.
A set $X\in 2^\omega$ is \emph{Martin-L\"of random} if $X\notin\bigcap_n U_n$ for every Martin-L\"of test.
\end{definition}

\begin{definition}[$1$-genericity]
A set $A\in 2^\omega$ is \emph{$1$-generic} if for every computably enumerable set of strings $W\subseteq 2^{<\omega}$
there exists $\sigma\prec A$ such that either $\sigma\in W$, or no extension of $\sigma$ belongs to $W$.
Equivalently, $A$ meets or avoids every effectively open set $[W]$.
\end{definition}

\begin{proposition}[Topological size and Borel complexity of $1$-generics]
\label{prop:1generic-comeager-Pi02}
Let $\mathcal{G}_1\subseteq 2^\omega$ be the class of $1$-generic sets. Then $\mathcal{G}_1$ is a dense
$G_\delta$ subset of $2^\omega$. In particular, $\mathcal{G}_1$ is comeager and (in the Borel hierarchy)
a $\Pi^0_2$ set.
\end{proposition}

\begin{proof}
Fix an effective enumeration $(W_e)_{e\in\omega}$ of all computably enumerable subsets of $2^{<\omega}$.
For each $e$ define
\[
\mathcal{D}_e :=
\Bigl\{X\in 2^\omega : \exists n\ \bigl(X{\upharpoonright}n\in W_e\bigr)\Bigr\}
\ \cup\
\Bigl\{X\in 2^\omega : \exists n\ \forall \tau\succeq X{\upharpoonright}n\ (\tau\notin W_e)\Bigr\}.
\]
Then $X\in\mathcal{D}_e$ iff $X$ meets or avoids $W_e$, hence
\[
\mathcal{G}_1=\bigcap_{e\in\omega}\mathcal{D}_e.
\]
Each $\mathcal{D}_e$ is open: if $X\in\mathcal{D}_e$ and $n$ witnesses membership, then the cylinder
$[X{\upharpoonright}n]$ is contained in $\mathcal{D}_e$.
Moreover, $\mathcal{D}_e$ is dense: given any $\sigma\in 2^{<\omega}$, either some extension of $\sigma$
lies in $W_e$ (hence $[\sigma]$ meets the first part), or no extension of $\sigma$ lies in $W_e$
(hence $[\sigma]$ is contained in the second part).
Therefore $\mathcal{G}_1$ is a countable intersection of dense open sets, i.e.\ a dense $G_\delta$ set.
\end{proof}

\begin{remark}
Here $\Pi^0_2$ refers to the \emph{Borel} hierarchy on $2^\omega$ (so $\Pi^0_2 = G_\delta$), not to the
arithmetical hierarchy of individual reals (e.g.\ $\Delta^0_2$).
\end{remark}

\section{The rigid trap}
\label{sec:rigid-trap}

In this section we prove the main structural result of the paper: $m$-rigidity alone forces infinite
antichains of $1$-degrees inside a fixed many-one degree. The proof is a uniform duplication argument
parametrized by computable sets $S\in\mathrm{REC}$.


\begin{theorem}[Rigidity trap]
\label{thm:mrigid-antichain}
If $A$ is $m$-rigid, then $\deg_m(A)$ contains an infinite antichain of $1$-degrees.
\end{theorem}

\begin{proof}
Fix $A$ $m$-rigid. By Lemma~\ref{lem:mrigid-coinfinite}, choose $c\in\omega$ with $c\notin A$.

For each computable set $S\subseteq\omega$, define the total computable function $f_S:\omega\to\omega$ by
\[
f_S(z)=
\begin{cases}
x & \text{if } z=2x,\\
x & \text{if } z=2x+1 \text{ and } x\in S,\\
c & \text{if } z=2x+1 \text{ and } x\notin S,
\end{cases}
\]
and let $B_S:=f_S^{-1}(A)$. Then $B_S\le_m A$ via $f_S$. Conversely, since $f_S(2x)=x$ for all $x$,
we have $x\in A \iff 2x\in B_S$, so $A\le_m B_S$ via $x\mapsto 2x$. Hence $B_S\equiv_m A$.
\medskip

\noindent\textbf{Key claim.} If $B_S\le_1 B_T$, then $S\subseteq^{*}T$.

\smallskip
Assume $B_S\le_1 B_T$ via a total computable injective $h$. Define
\[
k_0(x):= f_T(h(2x)),\qquad
k_1(x):=
\begin{cases}
f_T(h(2x+1)) & \text{if } x\in S,\\
x & \text{if } x\notin S.
\end{cases}
\]
Since $S$ is a computable set, $k_1$ is a total computable function.
We claim that $k_0$ is an $m$-autoreduction of $A$. Indeed, for every $x$,
\[
x\in A \iff 2x\in B_S \iff h(2x)\in B_T \iff f_T(h(2x))\in A \iff k_0(x)\in A.
\]
Similarly, for $k_1$, if $x\notin S$ then $k_1(x)=x$ and the equivalence is trivial, while if $x\in S$ then
$f_S(2x+1)=x$ and
\[
x\in A \iff 2x+1\in B_S \iff h(2x+1)\in B_T \iff f_T(h(2x+1))\in A \iff k_1(x)\in A.
\]

Hence both $k_0$ and $k_1$ are $m$-autoreductions of $A$.
By $m$-rigidity, both are eventually the identity. Hence for all but finitely many $x\in S$ we have
$f_T(h(2x))=x=f_T(h(2x+1))$. Injectivity of $h$ implies the fiber $f_T^{-1}(x)$ has at least two points.
For $x\neq c$ (note that $c$ is the only value that may have a large fiber under $f_T$), the definition of $f_T$
gives $f_T^{-1}(x)=\{2x\}$ if $x\notin T$ and $f_T^{-1}(x)=\{2x,2x+1\}$ if $x\in T$; hence $|f_T^{-1}(x)|\ge 2$
implies $x\in T$.
Therefore, for all but finitely many $x\in S\setminus\{c\}$ we have $x\in T$, and thus $S\subseteq^{*}T$.

\medskip
Now choose a family $(S_i)_{i\in\omega}$ of computable sets that is pairwise $\subseteq^{*}$-incomparable, e.g.
\[
S_i:=\{2^i(2n+1):n\in\omega\}.
\]
Then $B_{S_i}\not\le_1 B_{S_j}$ for $i\neq j$, so $\{B_{S_i}:i\in\omega\}$ is an infinite $\le_1$-antichain
contained in $\deg_m(A)$.
\end{proof}

\begin{remark}[A structural reading of the duplication construction]
\label{rem:quasi-embedding}
Fix an $m$-rigid set $A$ and define $(B_S)_{S\in\mathrm{REC}}$ as in the proof of
Theorem~\ref{thm:mrigid-antichain}. The argument yields more than the existence of an antichain: it shows that the map
\[
\Phi:\mathrm{REC}\to \mathcal{P}(\omega), \qquad \Phi(S):=B_S,
\]
is \emph{order-reflecting} with respect to almost inclusion. Namely, for all computable sets $S,T$,
\[
B_S \le_1 B_T \ \Longrightarrow\ S\subseteq^{*}T.
\]
Equivalently, if $S\not\subseteq^{*}T$ then $B_S \not\le_1 B_T$.
In particular, $\subseteq^{*}$-incomparability of parameters yields $\le_1$-incomparability of the corresponding
sets, which is the mechanism behind the infinite $1$-degree antichains inside $\deg_m(A)$.
\end{remark}

\section{Typicality: category and measure}
\label{sec:typicality}

\begin{proposition}[Rigidity holds almost surely]
\label{prop:mrigid-measure1}
Let $\mathcal{R}\subseteq 2^\omega$ be the class of $m$-rigid sets. Then $\mu(\mathcal{R})=1$.
Moreover, every Martin-L\"of random set is $m$-rigid.
\end{proposition}

\begin{proof}
Fix a total computable function $k:\omega\to\omega$ and consider
\[
\mathrm{Fix}(k):=\{A\in 2^\omega : \forall x\ (A(x)=A(k(x)))\}.
\]
If $k$ is eventually the identity, it does not violate $m$-rigidity, so we need not consider it.
Otherwise $E:=\{x:k(x)\neq x\}$ is infinite computable.
Choose a computable increasing sequence $(x_n)_{n\in\omega}\subseteq E$ such that
$x_n>\max\{x_i,k(x_i):i<n\}$. For each $n$ let
\[
U_n:=\{A\in 2^\omega:\forall i<n\ (A(x_i)=A(k(x_i)))\}.
\]
Then $(U_n)$ is a uniformly computable sequence of clopen sets, hence a uniformly c.e.\ sequence of open sets.
Notice that $U_0 = 2^\omega$, so $\mu(U_0) = 1$. For $n>0$, $U_n$ is obtained from $U_{n-1}$ by imposing the single new constraint $A(x_{n-1})=A(k(x_{n-1}))$.
Because $x_{n-1}>\max\{x_i,k(x_i):i<n-1\}$, the coordinate $x_{n-1}$ is fresh (it does not appear in any of the constraints defining $U_{n-1}$). Furthermore, since $x_{n-1}\in E$, we have $x_{n-1}\neq k(x_{n-1})$. Therefore, conditioning on $U_{n-1}$, the value $A(x_{n-1})$ is a completely free boolean variable, and the new equality exactly halves the measure, yielding $\mu(U_n)=2^{-n}$.
Finally, $\mathrm{Fix}(k)\subseteq\bigcap_n U_n$, hence $\mu(\mathrm{Fix}(k))=0$ and $(U_n)$ is a Martin-L\"of test covering $\mathrm{Fix}(k)$.

Now, if $A$ is not $m$-rigid then $A\in \mathrm{Fix}(k)$ for some total computable $k$ that is not eventually
the identity. Since there are only countably many such $k$, the set of non-$m$-rigid reals is a countable union
of measure-zero sets, hence null. Therefore $\mu(\mathcal{R})=1$. Finally, every Martin-L\"of random set avoids
each $\mathrm{Fix}(k)$ covered by a Martin-L\"of test, so it must be $m$-rigid.
\end{proof}

\begin{corollary}[An almost sure positive answer to Odifreddi's Question~5]
\label{cor:almost-sure-odifreddi}
For $\mu$-almost every $A\in 2^\omega$, the many-one degree $\deg_m(A)$ contains an infinite antichain of
$1$-degrees. In particular, every Martin-L\"of random set $A$ has this property.
\end{corollary}

\begin{proof}
By Proposition~\ref{prop:mrigid-measure1}, $\mu$-almost every $A$ is $m$-rigid (and every Martin-L\"of random $A$
is $m$-rigid). Apply Theorem~\ref{thm:mrigid-antichain}.
\end{proof}

\begin{remark}
Equivalently, the class of reals $A$ whose many-one degree fails to contain an infinite antichain of $1$-degrees
is a null set. In particular, any counterexample to Odifreddi's Question~5 (if it exists) can only be realized
by reals from a measure-zero subset of Cantor space.
\end{remark}
\begin{theorem}[Jockusch {\cite[Theorem~4.1]{Jockusch1980}}]
\label{cor:1generic-mrigid}
Every $1$-generic set is $m$-rigid.
\end{theorem}

\begin{corollary}
\label{cor:1generic-antichain}
If $A$ is $1$-generic, then $\deg_m(A)$ contains an infinite antichain of $1$-degrees.
\end{corollary}

\begin{proof}
By Theorem~\ref{cor:1generic-mrigid}, $A$ is $m$-rigid. Now apply Theorem~\ref{thm:mrigid-antichain}.
\end{proof}

\section{Conclusions}
\label{sec:conclusions}

We isolated the notion of $m$-rigidity and proved the following structural principle:
if $A$ is $m$-rigid, then the many-one degree $\deg_m(A)$ contains an infinite antichain of $1$-degrees
(Theorem~\ref{thm:mrigid-antichain}). The duplication construction yields an order-reflecting map
$S\mapsto B_S$ from $(\mathrm{REC},\subseteq^{*})$ into the $\le_1$-structure on a subfamily of $\deg_m(A)$
(Remark~\ref{rem:quasi-embedding}), which is the mechanism behind the antichain.

As consequences, Odifreddi's Question~5 has large families of positive instances in both standard senses of
typicality on Cantor space. By Jockusch's theorem, every $1$-generic set is $m$-rigid, yielding a comeager
family of degrees satisfying the conclusion. Moreover, $m$-rigidity holds with Lebesgue measure $1$
(Proposition~\ref{prop:mrigid-measure1}), so the conclusion holds \emph{almost surely} and in particular for
every Martin-L\"of random real (Corollary~\ref{cor:almost-sure-odifreddi}). Hence any counterexample, if it
exists, must be confined to a null set (and also to a meager set).

A natural direction for further work is to characterize how far $m$-rigidity can be weakened while still
forcing infinite $1$-antichains inside $\deg_m(A)$.

\section*{Acknowledgments}

This work is the result of an extended human–AI collaboration. Several of the main structural ideas and technical arguments emerged from exploratory interaction with AI-based reasoning systems: Gemini Deep Think (Google DeepMind) and ChatGPT Pro (OpenAI). The author has fully reworked and verified all arguments and bears sole responsibility for the correctness of the results.


\begin{thebibliography}{99}

\bibitem{Batyrshin2021}
I.\,I.~Batyrshin,
\emph{On $1$-degrees inside $m$-degrees},
Lobachevskii Journal of Mathematics \textbf{42} (2021), no.~12, 2740--2743.
doi:\,10.1134/S1995080221120076.

\bibitem{Degtev1973}
A.\,N.~Degtev,
\emph{$tt$- and $m$-degrees},
Algebra and Logic \textbf{12} (1973), no.~2, 78--89.
doi:\,10.1007/BF02219290.

\bibitem{Degtev1976}
A.\,N.~Degtev,
\emph{Partially ordered sets of $1$-degrees, contained in recursively enumerable $m$-degrees},
Algebra and Logic \textbf{15} (1976), no.~3, 153--164.
doi:\,10.1007/BF01876317.

\bibitem{DH2010}
R.~G.~Downey and D.~R.~Hirschfeldt,
\emph{Algorithmic Randomness and Complexity},
Theory and Applications of Computability,
Springer, 2010.

\bibitem{Jockusch1980}
C.\,G.~Jockusch~Jr.,
\emph{Degrees of generic sets},
in \emph{Recursion Theory: Its Generalisations and Applications}
(Proc.\ Logic Colloquium, Univ.\ Leeds, Leeds, 1979),
F.\,R.~Drake and S.\,S.~Wainer (eds.),
London Mathematical Society Lecture Note Series, vol.~45,
Cambridge University Press, Cambridge, 1980, pp.~110--139.

\bibitem{Nies2009}
A.~Nies,
\emph{Computability and Randomness},
Oxford Logic Guides, vol.~51,
Oxford University Press, 2009.


\bibitem{Odifreddi1981}
P.\,G.~Odifreddi,
\emph{Strong reducibilities},
Bull.\ Amer.\ Math.\ Soc.\ (N.S.) \textbf{4} (1981), no.~1, 37--86.
doi:\,10.1090/S0273-0979-1981-14863-1.

\bibitem{Odifreddi1999}
P.\,G.~Odifreddi,
\emph{Reducibilities},
in \emph{Handbook of Computability Theory},
E.\,R.~Griffor (ed.),
Studies in Logic and the Foundations of Mathematics, vol.~140,
Elsevier/North-Holland, Amsterdam, 1999, pp.~89--119.
doi:\,10.1016/S0049-237X(99)80019-6.

\bibitem{OdifreddiCRT}
P.~G.~Odifreddi,
\emph{Classical Recursion Theory},
Studies in Logic and the Foundations of Mathematics, vol.~125,
North-Holland, Amsterdam, 1989.

\bibitem{Soare1987}
R.~I.~Soare,
\emph{Recursively Enumerable Sets and Degrees: A Study of Computable Functions and Computably Generated Sets},
Perspectives in Mathematical Logic,
Springer-Verlag, Berlin, 1987.


\end{thebibliography}
\end{document}